\documentclass[10pt,leqno,letterpaper]{article}
\usepackage[utf8]{inputenc}
\usepackage{amsmath}
\usepackage{amsfonts}
\usepackage{amssymb}
\usepackage{graphicx}
\usepackage{mathrsfs}
\usepackage{upref,amsthm,amsxtra,exscale}
\usepackage{cite}
\usepackage[colorlinks=true,urlcolor=blue,
citecolor=red,linkcolor=blue,linktocpage,pdfpagelabels,
bookmarksnumbered,bookmarksopen]{hyperref}
\usepackage{fullpage}
\usepackage{upgreek}

\newtheorem{theorem}{Theorem}[section]

\newtheorem{remark}[theorem]{Remark}
\newtheorem{lemma}[theorem]{Lemma}
\newtheorem{proposition}[theorem]{Proposition}
\newtheorem{definition}[theorem]{Definition}
\newtheorem{examples}[theorem]{Examples}

\numberwithin{equation}{section}

\def\r{\mathbb{R}}
\def\rn{\mathbb{R}^N}
\def\z{\mathbb{Z}}

\def\n{\mathbb{N}}
\def\cc{\mathbb{C}}
\def\eps{\varepsilon}
\def\rh{\rightharpoonup}

\def\irn{\int_{\r^N}}
\def\vp{\varphi}
\def\o{\Omega}
\def\t{\Theta}
\def\bf{\mathbf}
\def\tilde{\widetilde}

\def\cC{\mathcal{C}}

\def\cH{\mathcal{H}}

\def\cJ{\mathcal{J}}

\def\cN{\mathcal{N}}

\def\cU{\mathcal{U}}

\def\supp{\mathrm{supp}}
\def\F{\mathrm{Fix}}

\author{Mónica Clapp\footnote{M. Clapp was supported by CONACYT (Mexico) through the grant for the research project A1-S-10457.} \ and \ Mayra Soares\footnote{M. Soares was supported by UNAM-DGAPA (Mexico) through a postdoctoral fellowship.}}
\title{Coupled and uncoupled sign-changing spikes of singularly perturbed elliptic systems}
\date{\today}

\begin{document}
	\maketitle
	
	\begin{abstract}
We study the existence and asymptotic behavior of solutions having positive and sign-changing components to the singularly perturbed system of elliptic equations
\begin{equation*}
		\begin{cases}
			-\eps^2\Delta u_i+u_i=\mu_i|u_i|^{p-2}u_i + \sum\limits_{\substack{j=1 \\ j \not=i}}^\ell\lambda_{ij}\beta_{ij}|u_j|^{\alpha_{ij}}|u_i|^{\beta_{ij} -2}u_i,\\
			u_i \in H^1_0(\o), \quad u_i\neq 0, \qquad i=1,\ldots,\ell,
		\end{cases} 
	\end{equation*}
in a bounded domain $\o$ in $\rn$, with $N\geq 4$, $\eps>0$, $\mu_i>0$, $\lambda_{ij}=\lambda_{ji}<0$, $\alpha_{ij}, \beta_{ij}>1$, $\alpha_{ij}=\beta_{ji}$, $\alpha_{ij} + \beta_{ij} = p\in (2,2^*)$, and $2^{*}:=\frac{2N}{N-2}$. 

If $\o$ is the unit ball we obtain solutions with a prescribed combination of positive and nonradial sign-changing components exhibiting two different types of asymptotic behavior as $\eps\to 0$: solutions whose limit profile is a rescaling of a solution with positive and nonradial sign-changing components of the limit system 
\begin{equation*}
		\begin{cases}
			-\Delta u_i+u_i=\mu_i|u_i|^{p-2}u_i + \sum\limits_{\substack{j=1 \\ j \not=i}}^\ell\lambda_{ij}\beta_{ij}|u_j|^{\alpha_{ij}}|u_i|^{\beta_{ij} -2}u_i,\\
			u_i \in H^1(\rn), \quad u_i\neq 0, \qquad i=1,\ldots,\ell,
		\end{cases}
	\end{equation*}
and solutions whose limit profile is a solution of the uncoupled system, i.e., after rescaling and translation, the limit profile of the $i$-th component is a positive or a nonradial sign-changing solution to the equation
$$-\Delta u+u=\mu_i|u|^{p-2}u,\qquad u \in H^1(\rn), \qquad u\neq 0.$$

		
	\end{abstract}
	
	\section{Introduction}
	\label{sec:introduction}
	
	We consider the following system of singularly perturbed elliptic equations
	\begin{equation} \label{P_e}
		\tag{$\mathscr S_{\eps,\o}$}\qquad
		\begin{cases}
			-\eps^2\Delta u_i+u_i=\mu_i|u_i|^{p-2}u_i + \sum\limits_{\substack{j=1 \\ j \not=i}}^\ell\lambda_{ij}\beta_{ij}|u_j|^{\alpha_{ij}}|u_i|^{\beta_{ij} -2}u_i,\\
			u_i \in H^1_0(\o), \quad u_i\neq 0, \qquad i=1,\ldots,\ell,
		\end{cases} 
	\end{equation}
	where $\eps>0$ is a small parameter, $\Omega$ is a bounded smooth domain in $\mathbb{R}^N$ which contains the origin, $N\geq 2$, $\mu_i>0$, $\lambda_{ij}=\lambda_{ji}<0$, $\alpha_{ij}, \beta_{ij}>1$, $\alpha_{ij}=\beta_{ji}$, $\alpha_{ij} + \beta_{ij} = p\in (2,2^*)$, and $2^{*}$ is the critical Sobolev exponent (i.e., $2^{*}:=\frac{2N}{N-2}$ if $N\geq 3$ and $2^{*}:=\infty$ if $N=2$).
	
	This system arises as a model for various physical phenomena, in particular in the study of standing waves for a mixture of Bose–Einstein condensates of $\ell$ different hyperfine states which overlap in space, see for example \cite{BBEG}. Here we consider the case in which the interaction between particles in the same state is attractive ($\mu_i>0$) and the interaction between particles in any two different states is repulsive ($\lambda_{ij}<0$). 
	
	In their seminal paper \cite{LW} Lin and Wei described the behavior of positive least energy solutions for the system \eqref{P_e} with cubic nonlinearity ($\alpha_{ij}=\beta_{ij}=2$) as $\eps\to 0$. They showed that each component is a spike, i.e., a rescaling of the positive ground state solution to the problem
	\begin{equation}\label{P_i}
		\tag{$\mathscr P_i$}
		\begin{cases}
			-\Delta u+u=\mu_i|u|^{p-2}u,\\
			u \in H^1(\rn), \quad u\neq 0,
		\end{cases}
	\end{equation}
	and that, as $\eps\to 0$, the centers of the spikes approach a sphere-packing position in $\o$, i.e., a configuration of points maximizing the distances among them and to the boundary of $\o$. Multiple positive solutions and a numerical description of them is given in \cite{WWL}.
	
	Our main objective is to study the existence and profile of solutions to \eqref{P_e} some of whose components can be positive while others change sign. It is reasonable to expect that there will be solutions with sign-changing spikes, i.e., solutions whose sign-changing components look like rescalings of a sign-changing solution to the problem \eqref{P_i}. The existence of nonradial sign-changing solutions to \eqref{P_i} was first established by Bartsch and Willem in \cite{BW} for dimensions $N=4$ and $N\geq6$, and by Lorca and Ubilla for $N=5$ in \cite{LU}, taking advantage of some symmetry properties of $\rn$. Other solutions of this type were found in \cite{CSr}.
	
	On the other hand, rescaling the components by $\widetilde u_i(x):=u_i(\eps x)$ the system \eqref{P_e} becomes
	\begin{equation*}
		\begin{cases}
			-\Delta u_i+u_i=\mu_i|u_i|^{p-2}u_i + \sum\limits_{\substack{j=1 \\ j \not=i}}^\ell\lambda_{ij}\beta_{ij}|u_j|^{\alpha_{ij}}|u_i|^{\beta_{ij} -2}u_i,\\
			u_i \in H^1_0(\o_\eps), \quad u_i\neq 0, \qquad i=1,\ldots,\ell,
		\end{cases} 
	\end{equation*}
	in $\o_\eps:=\{x\in\rn:\eps x\in\o\}$. As $\eps\to 0$ these domains cover the whole space $\rn$. So it is natural to ask whether the system \eqref{P_e} has a solution that, after rescaling, approaches a solution to the system
	\begin{equation}\label{Sl_i}
		\tag{$\mathscr S_{\infty,\ell}$}\quad
		\begin{cases}
			-\Delta u_i+u_i=\mu_i|u_i|^{p-2}u_i + \sum\limits_{\substack{j=1 \\ j \not=i}}^\ell\lambda_{ij}\beta_{ij}|u_j|^{\alpha_{ij}}|u_i|^{\beta_{ij} -2}u_i,\\
			u_i \in H^1(\rn), \quad u_i\neq 0, \qquad i=1,\ldots,\ell.
		\end{cases}
	\end{equation}
	As shown by Lin and Wei in \cite[Theorem 1]{LW2}, this system does not have a ground state solution. On the other hand, Sirakov showed in \cite{s} that it does have a positive least energy radial solution (i.e., every component $u_i$ is positive and radial); see also \cite{BjLyMsh,WjWy} and the references therein.
	
	One might expect to obtain solutions with positive and sign-changing components for the system \eqref{P_e} whose limit profile is a solution of the same type for the system \eqref{Sl_i}. For this last system with cubic nonlinearity (hence, $N\leq 3$) Sato and Wang \cite{sw2} established the existence of least energy solutions of this type whose components are radial; see also \cite{clz,LLW}. The following result provides solutions with nonradial sign-changing components. It is proved in Section \ref{sec:limit_system}.
	
	\begin{theorem} \label{thm:main_rn}
		Let $N=4$ or $N\geq 6$. Then, for any given $0\leq m\leq\ell$, the system \eqref{Sl_i} has a solution $\bf w=(w_1,\ldots,w_\ell)$ whose first $m$ components $w_1,\ldots,w_m$ are positive and whose last $\ell-m$ components $w_{m+1},\ldots,w_\ell$ are nonradial and change sign. Furthermore, $\bf w$ satisfies
		\begin{equation} \label{eq:symmetries}
			\begin{cases}
				w_i(z_1,z_2,x)=w_i(\mathrm{e}^{\mathrm{i}\vartheta} z_1,\mathrm{e}^{\mathrm{i}\vartheta} z_2,gx) &\text{for all \ }\vartheta\in[0,2\pi), \ g\in O(N-4), \ i=1,\ldots,\ell, \\
				w_i(z_1,z_2,x)=w_i(z_2,z_1,x) &\text{if \ }i=1,\ldots,m, \\
				w_i(z_1,z_2,x)=-w_i(z_2,z_1,x) &\text{if \ }i=m+1,\ldots,\ell,
			\end{cases}
		\end{equation}
		for all $(z_1,z_2,x)\in\cc\times\cc\times\r^{N-4}\equiv\rn$, and it has least energy among all nontrivial solutions with these symmetry properties.
	\end{theorem}	
	
	The symmetries of the sign-changing components $w_{m+1},\ldots,w_\ell$ are basically the same as those considered in \cite{BW}. So one may wonder whether one can obtain a similar result for $N=5$ using the symmetries introduced in \cite{LU} or \cite{CSr}. As we shall see, this is not possible. A key role is played by the space of fixed points of the group of symmetries involved, see Theorem \ref{thm:existence_rn}.
	
	For the singularly perturbed system \eqref{P_e} in a domain $\o$ having suitable symmetries we obtain solutions with a prescribed combination of positive and sign-changing components exhibiting two different types of asymptotic behavior as $\eps\to 0$. We get solutions whose limit profile is a rescaling of a solution to the limit system \eqref{Sl_i} with positive and sign-changing components, and solutions whose limit profile is a solution to the uncoupled system, i.e., the limit profile of the $i$-th component is a rescaling of a positive or a sign-changing solution to the problem \eqref{P_i}. To illustrate our results, let us focus on the case where $\o$ is the open unit ball $B_1(0)$ in $\rn$ centered at the origin. For $\eps>0$ and $u\in H^1(\rn)$ let
	\[\|u\|_\eps^2:=\frac{1}{\eps^N}\displaystyle\irn\Big[\eps^2|\nabla u|^2+u^2\Big]\qquad\text{and}\qquad\|u\|:=\|u\|_1.\]
	The following two theorems are special cases of Theorem \ref{thm:profiles}, which is stated and proved in Section \ref{sec:profiles}.
	
	\begin{theorem}\label{thm:main1}
		Let $N=4$ or $N\geq 6$, and $\o=B_1(0)$. Then, for any given $0\leq m\leq\ell$ and any sequence $(\eps_k)$ of positive numbers converging to zero, there exists a solution $\widehat{\bf u}_k=(\widehat u_{1k},\ldots,\widehat u_{\ell k})$ to the system $(\mathscr S_{\eps_k,\o})$ whose first $m$ components are positive and whose last $\ell-m$ components are nonradial and change sign, with the following limit profile: 
		
		There exists a fully nontrivial solution ${\bf w=(w_1,\ldots,w_\ell)}$ to the system \eqref{Sl_i} such that, after passing to a subsequence,
		\[\lim_{k \to \infty}\|\widehat u_{ik}-w_i(\eps_k^{-1} \ \cdot \ )\|_{\eps_k}=0\qquad\text{for all \ }i=1,\ldots,\ell.\]
		The first $m$ components of $\bf w$ are positive, its last $\ell-m$ components are nonradial and change sign, and $\bf w$ satisfies \eqref{eq:symmetries}. Therefore,
		\[\lim_{k\to\infty}\displaystyle\sum_{i=1}^\ell\|\widehat u_{ik}\|_{\eps_k}^2=\displaystyle\sum_{i=1}^\ell\|w_i\|^2=:\widehat {\mathfrak{c}}_m.\]
	\end{theorem}
	
	\begin{theorem}\label{thm:main2}
		Let $N\geq 5$ and $\o=B_1(0)$. Then, for any given $0\leq m\leq\ell$ and any sequence $(\eps_k)$ of positive numbers converging to zero, there exists a solution $\bf u_k=(u_{1k},\ldots,u_{\ell k})$ to the system $(\mathscr S_{\eps_k,\o})$ whose first $m$ components are positive and whose last $\ell-m$ components are nonradial and change sign, with the following limit profile: 
		
		For each $i=1,\ldots,\ell$, there exist a sequence $(\xi_{ik})$ in $B_1(0)$ and a nontrivial solution $v_i$ to the problem \eqref{P_i} such that, after passing to a subsequence,
		\[\lim_{k \to \infty}\eps_k^{-1} \mathrm{dist}(\xi_{ik},\partial B_1(0)) = \infty,\quad\lim_{k \to \infty}\eps_k^{-1}|\xi_{ik}-\xi_{jk}|=\infty\text{ if }i\neq j,\quad\displaystyle\lim_{k\to\infty}\|u_{ik}-v_{i}(\eps_k^{-1}( \ \cdot \ -\xi_{ik}))\|_{\eps_k}=0.\]
		The functions $v_1,\ldots,v_m$ are positive and radial, while the functions $v_{m+1},\ldots,v_\ell$ are sign-changing, nonradial and satisfy
		\begin{equation} \label{eq:symmetries2}
			\begin{cases}
				v_i(z_1,z_2,x)=v_i(\mathrm{e}^{\mathrm{i}\vartheta} z_1,\mathrm{e}^{\mathrm{i}\vartheta} z_2,gx) &\text{for all \ }\vartheta\in[0,2\pi), \ g\in O(N-4), \\
				v_i(z_1,z_2,x)=-v_i(z_2,z_1,x),
			\end{cases}
		\end{equation}
		for all $(z_1,z_2,x)\in\cc\times\cc\times\r^{N-4}\equiv\rn$, $i=m+1,\ldots,\ell$. Furthermore,
		\[\lim_{k\to\infty}\displaystyle\sum_{i=1}^\ell\|u_{ik}\|_{\eps_k}^2=\displaystyle\sum_{i=1}^\ell\|v_i\|^2=:\mathfrak{c}_m,\]
satisfies $\mathfrak{c}_m<\widehat{\mathfrak{c}}_m$, with $\widehat{\mathfrak{c}}_m$ as in \emph{Theorem \ref{thm:main1}}, if $N\geq 6$.
	\end{theorem}
	
	These last two results exhibit some interesting facts. The solutions given by Theorem \ref{thm:main2} are least energy solutions to the system \eqref{P_e} having some specific symmetries described in Example \ref{examples}$(ii)$ below. Their components behave as expected in the repulsive case $\lambda_{ij}<0$, i.e., they concentrate at points that are far from each other and from the boundary of the ball. On the other hand, the solutions given by Theorem \ref{thm:main1} behave as shown by Lin and Wei  \cite[Theorem 1.1(2)]{LW} for the attractive case $\lambda_{ij}>0$, i.e., all components concentrate at the origin. The solutions given by Theorem \ref{thm:main1} enjoy the same symmetries as those given by Theorem \ref{thm:main2} (and more), but they have higher energy. This shows that the sign of the interaction coefficient $\lambda_{ij}$ is not determinant in the segregation behavior of higher energy solutions.
	
	Another interesting feature of the solutions given by Theorem \ref{thm:main1} is that the limit profiles of their components are coupled, i.e., they solve the system \eqref{Sl_i}, up to rescaling. A similar behavior can be seen for least energy positive radial solutions to \eqref{P_e} in the unit ball: their limit profile is a positive least energy radial solution to the system \eqref{Sl_i}, see Theorem \ref{thm:radial}. As far as we know, these are the first examples in the literature that show this kind of asymptotic behavior.
	
	To prove our results we follow the approach introduced in \cite{CSr}. Using concentration compactness techniques, we carry out a careful analysis of the behavior of minimizing sequences for the system \eqref{Sl_i} that have a specific type of symmetries, described in Section \ref{sec:preliminaries}. The symmetries can be chosen to produce a change of sign by construction. As $\eps\to 0$ the components of the system concentrate at fixed points of the group action. If the space of fixed points is trivial, all components will necessarily concentrate at the origin, but they will move far away from each other when the fixed-point space has positive dimension. A detailed description is given in Theorem \ref{thm:spliting}.
	
	We note that for $N=3$ there are no symmetries with the properties required to produce nonradial sign-changing solutions. For the single equation \eqref{P_i} the existence of this kind of solutions was shown by Musso, Pacard and Wei in \cite{MPW} using the Lyapunov-Schmidt reduction procedure.

We wish to point out that solutions with sign-changing components for a system related to the existence of optimal partitions for the eigenvalue problem were obtained by Tavares and Terracini in \cite{tt}. For the system $(\mathscr S_{1,\o})$ with fixed $\eps=1$ and cubic nonlinearity (hence, $N\leq 3$) in a bounded domain $\o$, existence and multiplicity of sign-changing and semi-nodal solutions (some components change sign and others are positive) have been established by several authors, see, e.g., \cite{sw1, sw2, clz}.
		
This paper is organized as follows: in Section \ref{sec:preliminaries} we introduce the symmetric variational setting for systems with positive and sign-changing components and establish the existence of minimizers for the system \eqref{P_e}. In Section \ref{sec:limit_system} we give a detailed description of the symmetric minimizing sequences for the system \eqref{Sl_i} and prove Theorem \ref{thm:main_rn}. In Section	\ref{sec:profiles} we state and prove a general result for the system \eqref{P_e} and derive Theorems \ref{thm:main1} and \ref{thm:main2} from it.

	\section{Minimizers with positive and sign-changing components}
	\label{sec:preliminaries}
	
	Let $G$ be a closed subgroup of the group $O(N)$ of linear isometries of $\rn$ and denote by $Gx:=\{gx:g\in G\}$ the $G$-orbit of $x\in\rn$. Let $\phi:G\to\z_2:=\{-1,1\}$ be a continuous homomorphism of groups with the following property:
	\begin{itemize}
		\item[$(A_1)$] If $\phi$ is surjective, then there exists $x_0\in\rn$ such that $Kx_0\neq Gx_0$ where $K:=\ker\phi$.
	\end{itemize} 
	Let $\t$ be an open subset of $\rn$ which is \emph{$G$-invariant}, i.e., $Gx\subset\t$ for every $x\in\t$. Then, a function $u:\t\to\r$ is called \emph{$G$-invariant} if it is constant on $Gx$ for every $x\in\t$ and will be called \emph{$\phi$-equivariant} if
	$$ u(gx)=\phi(g)u(x) \text{ \ for all \ }g\in G, \ x\in\t.$$	
	Define 
	$$H_0^1(\t)^\phi:=\{u\in H_0^1(\t): u\text{ is }\phi\text{-equivariant}\}.$$
	Assumption $(A_1)$ guarantees that $H_0^1(\t)^\phi$ has infinite dimension, see \cite{JMW}. If $\phi\equiv 1$ is the trivial homomorphism, then $H_0^1(\t)^\phi$ is the space of $G$-invariant functions in $H_0^1(\t)$. On the other hand, if $\phi$ is surjective, then every nontrivial function $u\in H_0^1(\t)^\phi$ is nonradial and changes sign. The following examples are of relevance to our main results.
	
	\begin{examples} \label{examples}
		Consider the following examples of groups of isometries and homomorphims satisfying $(A_1)$:
		\begin{itemize}
			\item[$(i)$] For any group $G$ the trivial homomorphism $\phi\equiv 1$ is not surjective, so $(A_1)$ is trivially satisfied.
			\item[$(ii)$] Let $\Gamma$ be the group generated by $\{\mathrm{e}^{\mathrm{i}\vartheta}:\vartheta\in [0,2\pi)\} \cup \{\tau\}$ acting on $\rn$ by
			$$\mathrm{e}^{\mathrm{i}\vartheta}(z_1,z_2,x)=(\mathrm{e}^{\mathrm{i}\vartheta} z_1,\mathrm{e}^{\mathrm{i}\vartheta} z_2,x)\quad\text{and}\quad \tau(z_1,z_2,x)=(z_2,z_1,x)\qquad\forall \ (z_1,z_2,x)\in\cc\times\cc\times\r^{N-4}\equiv\rn,$$
			and let $\phi:\Gamma\to\z_2$ be the homomorphism given by $\phi(\mathrm{e}^{\mathrm{i}\vartheta}):=1$ and $\phi(\tau):=-1$. The kernel of $\phi$ is the group $K:=\{\mathrm{e}^{\mathrm{i}\vartheta}:\vartheta\in [0,2\pi)\}$ and the point $x_0=(1,0,0)\in\cc\times\cc\times\r^{N-4}$ is such that 
			$$Kx_0=\{(\mathrm{e}^{\mathrm{i}\vartheta},0,0):\vartheta\in [0,2\pi)\}\quad\text{and}\quad\Gamma x_0=\{(\mathrm{e}^{\mathrm{i}\vartheta},0,0):\vartheta\in [0,2\pi)\}\cup \{(0,\mathrm{e}^{\mathrm{i}\vartheta},0):\vartheta\in [0,2\pi)\},$$
			so $(A_1)$ is satisfied.
			\item[$(iii)$] Let $G:=\Gamma\times O(N-4)$ with $\Gamma$ as in $(ii)$ and $g\in O(N-4)$ acting as
			$$g(z_1,z_2,x)=(z_1,z_2,gx)\qquad\forall \ (z_1,z_2,x)\in\cc\times\cc\times\r^{N-4}\equiv\rn.$$
			Let $\phi:G\to \z_2$ be the homomorphism given by $\phi(\mathrm{e}^{\mathrm{i}\vartheta}):=1$, $\phi(\tau):=-1$ and $\phi(g):=1$ for $g\in O(N-4)$. Then $K:=\ker\phi=\{\mathrm{e}^{\mathrm{i}\vartheta}:\vartheta\in [0,2\pi)\}\times O(N-4)$ and, if $x_0=(1,0,0)\in\cc\times\cc\times\r^{N-4}$,
			$$Kx_0=\{(\mathrm{e}^{\mathrm{i}\vartheta},0,0):\vartheta\in [0,2\pi)\}\quad\text{and}\quad Gx_0=\{(\mathrm{e}^{\mathrm{i}\vartheta},0,0):\vartheta\in [0,2\pi)\}\cup \{(0,\mathrm{e}^{\mathrm{i}\vartheta},0):\vartheta\in [0,2\pi)\},$$
			so $(A_1)$ is satisfied.
		\end{itemize}
	\end{examples}
	
	For the remaining of this section we fix a closed subgroup $G$ of $O(N)$ and, for each $i=1,\ldots,\ell$, a continuous homomorphism $\phi_i:G\to\z_2$ satisfying $(A_1)$.
	
	If $\t$ is a $G$-invariant open subset of $\rn$, we consider the system
	\begin{equation}\label{S_ept}
		\tag{$\mathscr S^\phi_{\eps,\Theta}$} \qquad
		\begin{cases} 
			-\eps^2\Delta u_i+u_i=\mu_i|u_i|^{p-2}u_i + \displaystyle\sum\limits_{\substack{j=1 \\ j \not=i}}^\ell\lambda_{ij}\beta_{ij}|u_j|^{\alpha_{ij}}|u_i|^{\beta_{ij} -2}u_i,\\
			u_i \in H^1_0(\Theta)^{\phi_i}, \quad u_i\neq 0, \qquad i=1,\ldots,\ell,
		\end{cases}
	\end{equation}	 
	with $\eps>0$, $\mu_i>0$, $\lambda_{ij}=\lambda_{ji}<0$, $\alpha_{ij}, \beta_{ij}>1$, $\alpha_{ij}=\beta_{ji}$ and $\alpha_{ij} + \beta_{ij} = p\in (2,2^*)$.
	
		As usual, we consider $H^1_0(\t)$ as a subspace of $H^1(\rn)$, identifying a function $u\in H^1_0(\t)$ with its trivial extension to $\rn$. Set
	\[ \cH^\ell(\t):=H_0^1(\t)^{\phi_1}\times\cdots\times H_0^1(\t)^{\phi_\ell},
	\]
	and denote an element in $\cH^\ell(\t)$ by $\bf u=(u_1,\ldots,u_\ell)$. For each $\eps>0$ we define
	\[
	\|\bf u\|_{\ell,\eps}:= \left(\displaystyle\sum_{i=1}^\ell\|u_i\|_\eps^2\right)^{1/2}, \quad \text{where \ }
	\|u_i\|_\eps^2 := \frac{1}{\eps^N}\displaystyle\irn\Big[\eps^2|\nabla u_i|^2 + u_i^2\Big].
	\]
	If $\eps=1$ we write $\|\bf u\|_{\ell}$ instead of $\|\bf u\|_{\ell,1}$. Note that $\|\cdot\|_\eps$ is a norm in $H_0^1(\t)$ equivalent the standard one $\|\cdot\|:=\|\cdot\|_1$. Therefore, $\|\cdot\|_{\ell,\eps}$ is a norm in $\cH^\ell(\t)$ for every $\eps>0$ and all of these norms are equivalent.
	
	Consider the functional $\cJ^\ell_\eps : \cH^\ell(\t) \to \r$ given by
	\[
	\cJ^\ell_\eps(\bf u) := \dfrac{1}{2}\displaystyle\sum_{i=1}^\ell\|u_i\|_\eps^2 - \dfrac{1}{p}\displaystyle\sum_{i=1}^\ell\dfrac{1}{\eps^N}\int_{\mathbb{R}^N}\mu_i|u_i|^p - \dfrac{1}{2}\displaystyle\sum\limits_{\substack{i,j=1 \\ j\not=i}}^\ell\dfrac{1}{\eps^N}\int_{\mathbb{R}^N}\lambda_{ij}|u_j|^{\alpha_{ij}}|u_i|^{\beta_{ij}},
	\]
	which is of class $\mathcal{C}^1$. Since $\lambda_{ij} = \lambda_{ji}$, $\beta_{ij}=\alpha_{ji}$ and $\alpha_{ij}+\beta_{ij}=p$, its partial derivatives are given by
	\begin{align} \label{dJ}
		\partial_i\cJ^\ell_\eps(\bf u)v &= \dfrac{1}{\eps^N}\int_{\mathbb{R}^N}\Big(\eps^2\nabla u_i \cdot \nabla v + u_iv\Big) - \dfrac{1}{\eps^N}\int_{\mathbb{R}^N}\mu_i|u_i|^{p-2}u_iv \nonumber\\
		&\qquad- \dfrac{1}{2}\displaystyle\sum\limits_{\substack{j=1 \\ j \not=i}}^\ell\dfrac{1}{\eps^N}\int_{\mathbb{R}^N}\lambda_{ij}\beta_{ij}|u_j|^{\alpha_{ij}}|u_i|^{\beta_{ij}-2}u_iv - \dfrac{1}{2}\displaystyle\sum_{\substack{j=1\\j\not=i}}^\ell\dfrac{1}{\eps^N}\int_{\mathbb{R}^N}\lambda_{ij}\alpha_{ji}|u_i|^{\alpha_{ji}-2}u_iv|u_j|^{\beta_{ji}}\nonumber\\
		&=\frac{1}{\eps^N}\Big[\displaystyle\irn(\eps^2\nabla u_i \cdot \nabla v + u_iv) - \displaystyle\int_{\mathbb{R}^N}\mu_i|u_i|^{p-2}u_iv - \displaystyle\sum_{\substack{j=1\\j\not=i}}^\ell\int_{\mathbb{R}^N}\lambda_{ij}\beta_{ij}|u_j|^{\alpha_{ij}}|u_i|^{\beta_{ij}-2}u_iv\Big],
	\end{align}
	for $v \in H^1_0(\t)^{\phi_i}$ and $i=1,\ldots,\ell.$ So, by the principle of symmetric criticality, the solutions to system \eqref{S_ept} are the critical points of $\cJ^\ell_\eps$ whose components $u_i$ are nontrivial. They belong to the Nehari-type set
	\begin{equation*}
		\cN^\ell_{\eps}(\t) := \Big\{\bf u\in \cH^\ell(\t) : u_i \neq 0, \ \partial_i\cJ^\ell_\eps(\bf u)u_i=0, \ \forall \ i = 1,\ldots,\ell\Big\}.
	\end{equation*}
	Note that
	\begin{equation} \label{eq:nehari}
		\partial_i\cJ^\ell_\eps(\bf u)u_i=\|u_i\|^2_\eps -  \frac{1}{\eps^N}\int_{\mathbb{R}^N}\mu_i|u_i|^p - \displaystyle\sum_{\substack{j=1\\j\not=i}}^\ell\frac{1}{\eps^N}\int_{\mathbb{R}^N}\lambda_{ij}\beta_{ij}|u_j|^{\alpha_{ij}}|u_i|^{\beta_{ij}}.
	\end{equation}
	Define 
	\begin{equation*}
		c_\eps^\ell(\t) := \inf_{\bf u\in \cN^\ell_{\eps}(\t)}\cJ^\ell_\eps(\bf u).
	\end{equation*}
	From \eqref{dJ} one sees that
	\begin{equation} \label{eq:J_nehari}
		\cJ^\ell_\eps(\bf u) = \dfrac{p-2}{2p}\displaystyle\sum_{i=1}^\ell \|u_i\|_\eps^2 = \dfrac{p-2}{2p}\|\bf u\|^2_{\ell,\eps} \qquad \text{if} \ \bf u = (u_1,u_2,\ldots,u_\ell) \in \cN^\ell_{\eps}(\t).
	\end{equation}
	
	\begin{remark} \label{rem:norms}
		\emph{For $\bf u=(u_1,\ldots,u_\ell)\in \cH^\ell(\t)$ and $\eps>0$, define $\tilde u_i(z):=u_i(\eps z)$. It is straightforward to verify that $\tilde u_i\in H^1(\rn)^{\phi_i}$ and
			\[
			\|u_i\|^2_\eps=\|\tilde u_i\|^2,\qquad\frac{1}{\eps^N}\displaystyle\irn|u_i|^p=\displaystyle\irn|\tilde u_i|^p,\qquad\frac{1}{\eps^N}\displaystyle\irn|u_j|^{\alpha_{ij}}|u_i|^{\beta_{ij}}=\displaystyle\irn |\tilde u_j|^{\alpha_{ij}}|\tilde u_i|^{\beta_{ij}}.\]
		}
	\end{remark}
	
	For $\bf u=(u_1,\ldots,u_\ell)\in\cH^\ell(\t)$ and $\bf t=(t_1,\ldots,t_\ell)\in(0,\infty)^\ell$, we write 
	$\bf t\bf u:=(t_1u_1,\ldots,t_\ell u_\ell).$
	Our next results provide some useful information about the Nehari-type set $\cN^\ell_\eps(\t)$.
	
	\begin{lemma}\label{N}
		The following statements hold true:
		\begin{itemize}
			\item[$(i)$]$\cN^\ell_\eps(\t)\not=\emptyset$ and there exists a constant $c_0>0$, independent of $\eps$, such that
			\[
			\|u_i\|_\eps^2 > c_0\quad\text{for every \ }\bf u = (u_1,\ldots,u_\ell) \in \cN^\ell_\eps(\t), \ i=1,\ldots,\ell.
			\]
			Therefore $\cN^\ell_\eps(\t)$ is a closed subset of $\cH^\ell(\t)$ and $0<c_\eps^\ell(\t) <\infty$.
			\item[$(ii)$]For each $\bf u \in \cH^\ell(\t)$ such that
			\begin{equation} \label{cft}
				\displaystyle\irn\mu_i|u_i|^p +\displaystyle\sum_{\substack{j=1\\j\not=i}}^\ell\int_{\mathbb{R}^N}\lambda_{ij}\beta_{ij}|u_j|^{\alpha_{ij}}|u_i|^{\beta_{ij}}>0 \quad \forall \ i=1,\ldots,\ell,
			\end{equation}
			there exists a unique $\bf t_{\bf u} \in (0,\infty)^{\ell},$ such that $\bf t_{\bf u}\bf u \in \cN^\ell_\eps(\t)$. Furthermore,  
			\[
			\cJ^\ell_\eps(\bf t_{\bf u} \bf u) = \max_{{\bf t} \in (0,\infty)^\ell}\cJ^\ell_\eps(\bf t \bf u).
			\]
		\end{itemize}
	\end{lemma}	
	
	\begin{proof}
		$(i)$ Since $\phi_i$ satisfies $(A_1)$, for each $i=1,\ldots,\ell$ one can choose $u_i\in \cC_c^\infty(\t)\cap H_0^1(\t)^{\phi_i}$, $u_i\not= 0$, such that $\supp(u_i)\displaystyle\cap\supp(u_j) = \emptyset$ for all $j\not=i$, and then take $t_i>0$ such that $v_i:=t_iu_i$ satisfies $\|v_i\|_\eps^2 = \eps^{-N}\displaystyle\irn\mu_i|v_i|^p$. Since ${\supp(v_i)\cap \supp(v_j)=\emptyset}$ for all $j\not=i$, setting ${\bf v := (v_1,\ldots,v_\ell)}$ we get that $\partial_i\cJ^\ell_\eps(\bf v)v_i=0$ and $v_i\not=0$ for all $i=1,\ldots,\ell$. Therefore $\bf v \in \cN^\ell_\eps(\t)$. 
		
		Let $\bf u = (u_1,\ldots,u_\ell) \in \cN^\ell_\eps(\t)$ and set $\tilde u_i(z):=u_i(\eps z)$. Then, since $\lambda_{ij}<0$ for every pair $i\not=j$, from Remark \ref{rem:norms} and Sobolev's embedding theorem we derive 
		\[
		\|\tilde u_i\|^2\leq\displaystyle\displaystyle\irn\mu_i|\tilde u_i|^p\leq C\mu_i\|\tilde u_i\|^p,
		\]
		with $C$ a positive constant, depending only on $N$ and $p$. This yields the estimate in statement $(i)$ and completes its proof.
		
		$(ii)$ Fix $\bf u=(u_1,\ldots,u_\ell)\in\cH^\ell(\t)$ and define $J^\ell_{\bf u}:(0,\infty)^\ell \to\r$ by
		\[
		J^\ell_{\bf u}(\bf t):=\cJ^\ell_\eps(\bf t\bf u) = \displaystyle\sum_{i=1}^\ell a_{u,i}t_i^2 - \displaystyle\sum_{i=1}^\ell b_{u,i}t_i^p + \displaystyle\sum\limits_{\substack{i,j=1 \\ j \neq i}}^\ell d_{u,ij} t_j^{\alpha_{ij}}t_i^{\beta_{ij}},
		\]
		where $\bf t=(t_1,\ldots,t_\ell)\in(0,\infty)^\ell$, 
		\[a_{u,i}:=\frac{1}{2}\|u_i\|^2_\eps,\quad b_{u,i}:=\frac{1}{p\eps^N}\displaystyle\irn\mu_i|u_i|^p,\quad d_{u,ij}:=\frac{1}{2\eps^N}\displaystyle\sum_{\substack{j=1\\j\neq i}}^\ell\displaystyle\irn|\lambda_{ij}|\,\beta_{ij}|u_j|^{\alpha_{ij}}|u_i|^{\beta_{ij}}.
		\]
		If $u_i\neq 0$ for all $i=1,\ldots,\ell$, then, as
		\[
		t_i\,\partial_i J^\ell_{\bf u}(\bf t)=\partial_i\cJ^\ell_\eps(\bf t\bf u)[t_iu_i],\qquad i=1,\ldots,\ell,
		\]
		we have that $\bf t\bf u\in\cN^\ell(\t)$ iff  $\bf t$ is a critical point of $J^\ell_{\bf u}$. Note that \eqref{cft} implies that $u_i\neq 0$ for all $i=1,\ldots,\ell$. Hence, $(ii)$ follows from \cite[Lemmas 2.1 and 2.2]{CSz}.
	\end{proof} 
	
	\begin{theorem}\label{ES}
		If there exists $\bf u\in\cN^\ell_\eps(\t)$ such that $\cJ^\ell_\eps(\bf u)=c_\eps^\ell(\t)$, then $\bf u=(u_1,\ldots,u_\ell)$ is a solution to  system \eqref{S_ept}. 
	\end{theorem}
	
	\begin{proof}
		The proof is practically the same as that of \cite[Theorem 3.4(a)]{CSz} with the obvious modifications.
	\end{proof}
	
	\begin{definition}
		A solution to the system \eqref{S_ept} such that $\cJ^\ell_\eps(\bf u)=c_\eps^\ell(\t)$ will be called a least energy solution to \eqref{S_ept}. 
	\end{definition}	
	
	\begin{remark} \label{rem:positive}
		\emph{If $\phi_i\equiv 1$ for every $i$ in some subset $I$ of $\{1,\ldots,\ell\}$ and $\bf u$ is a least energy solution to \eqref{S_ept}, then, setting $v_i:=|u_i|$ if $i\in I$ and $v_i:=u_i$ if $i\not\in I$, we have that $v_i\in H^1_0(\t)^{\phi_i}$, \ $\bf v=(v_1,\ldots,v_\ell)\in\cN^\ell_\eps(\t)$ and $\cJ^\ell_\eps(\bf v)=\cJ^\ell_\eps(\bf u)$. Hence, $\bf v$ is a least energy solution to \eqref{S_ept} whose $i$-th component is positive for every $i\in I$.}
	\end{remark}
	
	\begin{theorem}[Existence of minimizers] \label{thm:minimizers} 
		If $\t$ is a bounded $G$-invariant domain in $\rn$, then, for each $\eps >0$, the system \eqref{S_ept}
		has a least energy solution.
	\end{theorem}
	
	\begin{proof} 
		Fix $\eps>0$ and let $\bf u_k=(u_{1k},\ldots,u_{\ell k})\in\cN^\ell_\eps(\t)$ satisfy $\cJ_\eps^\ell(\bf u_k)\to c_\eps^\ell(\t)$. From \eqref{eq:J_nehari} we see that $(u_{ik})$ is bounded in $H_0^1(\t)$ for every $i=1,\ldots,\ell$. So, after passing to a subsequence, $u_{ik}\rh u_i$ weakly in $H_0^1(\t)^{\phi_1}$ and $u_{ik}\to u_i$ strongly in $L^p(\t)$. From \eqref{eq:nehari} and Lemma \ref{N} we get that
		$$0<c_0\leq \|u_{ik}\|^2_\eps = \frac{1}{\eps^N}\int_{\mathbb{R}^N}\mu_i|u_{ik}|^p + \displaystyle\sum_{\substack{j=1\\j\not=i}}^\ell\frac{1}{\eps^N}\int_{\mathbb{R}^N}\lambda_{ij}\beta_{ij}|u_{jk}|^{\alpha_{ij}}|u_{ik}|^{\beta_{ij}}$$
		and, passing to the limit, we obtain
		$$0<c_0\leq\displaystyle\int_{\mathbb{R}^N}\mu_i|u_i|^p + \displaystyle\sum_{\substack{j=1\\j\not=i}}^\ell\int_{\mathbb{R}^N}\lambda_{ij}\beta_{ij}|u_j|^{\alpha_{ij}}|u_i|^{\beta_{ij}}\qquad\text{for all \ }i=1,\ldots,\ell.$$
		Set $\bf u:=(u_1,\ldots,u_\ell)$. By Lemma \ref{N} there exists $\bf t=(t_1,\ldots,t_\ell)\in(0,\infty)^\ell$ such that $\bf t\bf u\in\cN^\ell_\eps(\t)$ and
		\begin{align*}
			c_\eps^\ell(\t) &\leq \cJ^\ell_\eps(\bf t\bf u) \leq\displaystyle\liminf_{k\to\infty}\cJ^\ell_\eps(\bf t\bf u_k) \leq\displaystyle\lim_{k\to \infty}\cJ^\ell_\eps(\bf u_k)= c_\eps^\ell(\t),		
		\end{align*}
		because $\bf u_k \in \cN^\ell_\eps(\t)$ for all $k\in \mathbb{N}$. Thus, $\lim\limits_{k\to\infty}\cJ^\ell_\eps(\bf t\bf u_k)=\cJ^\ell_\eps(\bf t\bf u)$. As $u_{ik}\to u_i$ strongly in $L^p(\t)$, this implies that $\lim\limits_{k\to\infty}\|\bf t\bf u_k\|_{\ell,\eps}^2=\|\bf t\bf u\|_{\ell,\eps}^2$. Since $\bf u_k\rh \bf u$ weakly in $\cH^\ell(\t)$, we derive that $\bf u_k\to\bf u$ strongly in $\cH^\ell(\t)$. As a consequence, $\bf u\in\cN^\ell_\eps(\t)$ and $\cJ^\ell_\eps(\bf u) = c_\eps^\ell(\t)$.
	\end{proof}

	\section{The limit system}	
	\label{sec:limit_system}
	
	We fix a closed subgroup $G$ of $O(N)$ and, for each $i=1,\ldots,\ell$, a continuous homomorphism $\phi_i:G\to\z_2$. We assume throughout this section that the $\phi_i$ satisfy $(A_1)$ and that $G$ has the following property:
	\begin{itemize}
		\item[$(A_2)$] For every $x\in\rn$, the $G$-orbit of $x$ is either infinite, or $Gx=\{x\}$.
	\end{itemize}
	
	Our aim is to study the behavior of minimizing sequences for the functional related to the system
	\begin{equation}\label{Sl_pi}
		\tag{$\mathscr S_{\infty,\ell}^\phi$}\qquad
		\begin{cases}
			-\Delta u_i+u_i=\mu_i|u_i|^{p-2}u_i + \displaystyle\sum\limits_{\substack{j=1 \\ j \not=i}}^\ell\lambda_{ij}\beta_{ij}|u_j|^{\alpha_{ij}}|u_i|^{\beta_{ij} -2}u_i,\\
			u_i \in H^1(\rn)^{\phi_i}, \quad u_i\neq 0, \qquad 1<i\leq \ell.
		\end{cases}
	\end{equation}
	We write $\|\cdot\|$ for the standard norm in $H^1(\rn)$. Set $\cH^\ell:=H^1(\rn)^{\phi_1}\times\cdots\times H^1(\rn)^{\phi_\ell}$, \ ${\cJ^\ell_\infty:=\cJ_1^\ell:\cH^\ell\to\r}$ \ and \ $\cN^\ell_\infty:=\cN_1^\ell(\rn)$ as in the previous section, and let
	$$c^\ell_\infty:= \inf_{\bf u\in\cN^\ell_\infty}\cJ^\ell_\infty(\bf u).$$
	For each $i=1,\ldots,\ell$, consider the problem
	\begin{equation}\label{P_pi}
		\tag{$\mathscr P_i^{\phi_i}$}\qquad
		\begin{cases}
			-\Delta u+u=\mu_i|u|^{p-2}u,\\
			u \in H^1(\rn)^{\phi_i}, \quad u\neq 0.
		\end{cases}
	\end{equation}
	Its solutions are the critical points of the functional $J_i:H^1(\rn)^{\phi_i}\to\r$ given by
	$$J_i(u):=\frac{1}{2}\displaystyle\irn(|\nabla u|^2+u^2)-\frac{1}{p}\displaystyle\irn|u|^p,$$
	belonging to the Nehari manifold $\cN_i:=\{u\in H^1(\rn)^{\phi_i}:u\neq 0, \ J_i'(u)u=0\}$. We set
	$$c_i:= \inf_{u\in\cN_i}J_i(u).$$
	A solution $\bf u$ to \eqref{Sl_pi} satisfying $\cJ_\infty^\ell(\bf u)=c_\infty^\ell$ will be called a \emph{least energy solution to \eqref{Sl_pi}}. Similarly, a solution $u$ to \eqref{P_pi} satisfying $J_i(u)=c_i$ will be called a \emph{least energy solution to \eqref{P_pi}}. 
	
	The \emph{$G$-fixed-point space}
	$$\F(G):=\{x\in\rn:gx=x\text{ for all }g\in G\}$$
	plays an important role on the existence or nonexistence of a least energy solution to the system \eqref{Sl_pi}. We start with the following result.
	
	\begin{proposition}\label{prop:energy_estimates}
	The following statements hold true:
		\begin{itemize}
			\item[$(i)$] $c_\infty^\ell\geq\displaystyle\displaystyle\sum_{i=1}^\ell c_i.$
			\item[$(ii)$] If $\F(G)$ has positive dimension, then $c_\infty^\ell=\displaystyle\displaystyle\sum_{i=1}^\ell c_i$.
			\item[$(iii)$] If $\ell\geq 2$ and $c_\infty^\ell=\displaystyle\displaystyle\sum_{i=1}^\ell c_i$, then $c_\infty^\ell$ is not attained.
		\end{itemize}
	\end{proposition}
	
	\begin{proof}
		$(i):$ Let $\bf u=(u_1,\ldots,u_\ell)\in\cN_\infty^\ell$. As $\lambda_{ij}<0$, from \eqref{eq:nehari} and Lemma \ref{N} $(i)$ we get that
		$$0< \|u_i\|^2\leq\displaystyle\irn\mu_i|u_i|^p \qquad\text{for all \ }i=1,\ldots,\ell.$$
		Hence, there exist $t_1,\ldots,t_\ell\in(0,\infty)$ such that $t_iu_i\in\cN_i$. Setting $\bf t:=(t_1,\ldots,t_\ell)$ and applying Lemma \ref{N} $(ii)$ we obtain
		$$\displaystyle\sum_{i=1}^\ell c_i\leq \displaystyle\sum_{i=1}^\ell J_i(t_iu_i)\leq\cJ^\ell_\infty(\bf t\bf u) \leq \cJ^\ell_\infty(\bf u),$$
		as claimed.
		
		$(ii):$ For each $i=1,\ldots,\ell$, let $v_i\in\cN_i$. As $\F(G)$ has positive dimension we may choose $\xi_{ik}\in\F(G)$ such that $|\xi_{ik}-\xi_{jk}|\to\infty$ as $k\to\infty$ if $i\neq j$. Define $\bf u_k=(u_{1k},\ldots,u_{\ell k})$ by $u_{ik}(x):=v_i(x-\xi_{ik})$. Note that, as $\xi_{ik}\in\F(G)$, we have that $u_{ik}\in H^1(\rn)^{\phi_i}$. Furthermore, $u_{ik}\in\cN_i$, $J_i(u_{ik})=J_i(v_i)$ and
		$$\displaystyle\irn\lambda_{ij}\beta_{ij}|u_{jk}|^{\alpha_{ij}}|u_{ik}|^{\beta_{ij}}=\displaystyle\irn\lambda_{ij}\beta_{ij}|v_j( \ \cdot \ +\xi_{ik}-\xi_{jk})|^{\alpha_{ij}}|v_i|^{\beta_{ij}}=o(1).$$
		Therefore,
		\begin{align*}
			0<\frac{2p}{p-2} c_i\leq \|u_{ik}\|^2=\displaystyle\irn\mu_i|u_{ik}|^p +\displaystyle\sum_{\substack{j=1\\j\not=i}}^\ell\displaystyle\irn\lambda_{ij}\beta_{ij}|u_{jk}|^{\alpha_{ij}}|u_{ik}|^{\beta_{ij}} + o(1).
		\end{align*}
		Applying Lemma \ref{N} $(ii)$ we see that, for $k$ large enough, there exists $\bf t_k:=(t_{1k},\ldots,t_{\ell k})\in(0,\infty)^\ell$ such that $\bf t_k\bf u_k\in\cN_\infty^\ell$, and $t_{ik}\to 1$ for all $i=1,\ldots,\ell$. Hence,
		\begin{align*}
			c_\infty^\ell &\leq\cJ_\infty^\ell(\bf t_k\bf u_k)=\displaystyle\sum_{i=1}^\ell J_i(t_{ik}u_{ik})+o(1)=\displaystyle\sum_{i=1}^\ell J_i(v_i)+o(1).
		\end{align*}
		This shows that $\displaystyle\displaystyle\sum\limits_{i=1}^\ell c_i\geq c_\infty^\ell$ and from $(i)$ we obtain $\displaystyle\displaystyle\sum\limits_{i=1}^\ell c_i=c_\infty^\ell$.
		
		$(iii):$ Let $c_\infty^\ell=\displaystyle\sum\limits_{i=1}^\ell c_i$. To show that this value is not attained, we argue by contradiction. Assume that $\bf u=(u_1,\ldots,u_\ell)\in\cN^\ell_\infty$ and $\cJ^\ell_\infty(\bf u)=c_\infty^\ell$. There are two possibilities. If for every $i$ there exists $j\neq i$ such that $\displaystyle\displaystyle\displaystyle\irn| u_j|^{\alpha_{ij}}|u_i|^{\beta_{ij}}\neq 0$, then $\displaystyle \|u_i\|^2<\displaystyle\irn\mu_i|u_i|^p$ for all $i=1,\ldots,\ell$  and, hence, there exists $t_i\in(0,1)$ such that $t_iu_i\in\cN_i$ and \[
		c_i \leq J_i(t_iu_i)=\dfrac{p-2}{2p}\|t_iu_i\|^2<\dfrac{p-2}{2p}\|u_i\|^2.
		\]
		It follows that 
		$$c_\infty^\ell=\cJ^\ell_\infty(\bf u)=\frac{p-2}{2p}\displaystyle\sum_{i=1}^\ell\|u_i\|^2>\displaystyle\sum_{i=1}^\ell c_i,$$
		contradicting our assumption. On the other hand, if there exists $i$ such that $\displaystyle\displaystyle\displaystyle\irn |u_j|^{\alpha_{ij}}|u_i|^{\beta_{ij}}= 0$ for all $j\neq i$, then 
		$$\|u_i\|^2=\displaystyle\irn\mu_i|u_i|^p\qquad\text{and}\qquad J_i(u_i)=c_i.$$
		Hence, $u_i$ is a nontrivial solution to the problem
		$$-\Delta w + w = \mu_i|w|^{p-2}w,\qquad w\in H^1(\rn).$$
		But $\displaystyle\displaystyle\irn |u_j|^{\alpha_{ij}}|u_i|^{\beta_{ij}}= 0$ also implies that $|u_j|^{\alpha_{ij}}|u_i|^{\beta_{ij}}= 0$ a.e. in $\rn$. As $u_j\neq 0$ for all $j$, we have that $u_i=0$ in some subset of positive measure of $\rn$. This contradicts the unique continuation principle.
	\end{proof}
	
	The following lemma will play an important role in the proof of Theorem \ref{thm:spliting}.
	
	\begin{lemma}\label{lem:orbits} 
		For any given sequence $(y_k)$ in $\rn$, there exist a sequence $(\zeta_k)$ in $\rn$ and a constant $C>0$ such that, up to a subsequence,
		$$\mathrm{dist}(Gy_k,\zeta_k)\leq C\qquad\text{for all \ }k\in\n,$$
		and
		\begin{itemize}
			\item[•] either $\zeta_k\in\F(G)$ for all $k\in\n$,
			\item[•] or there exist $g_n\in G$, $n\in\n$, such that $\lim\limits_{k\to\infty}|g_{n_1}\zeta_k-g_{n_2}\zeta_k|=\infty$ for any pair $n_1,n_2\in\n$ with $n_1\neq n_2$.
		\end{itemize}
	\end{lemma}
	
	\begin{proof}
		This is the statement of \cite[Lemma 3.2]{CC} for a group $G$ satisfying $(A_2)$.
	\end{proof}
	
	We are ready to prove the main result of this section.
	
	\begin{theorem} \label{thm:spliting} Let $\ell\geq2$ and $\bf u_k=(u_{1k},\ldots,u_{\ell k})\in\cN_\infty^\ell$ be such that $\cJ_\infty^\ell(\bf u_k)\to c_\infty^\ell$. 
		\begin{itemize}
			\item[$(I)$] If $\F(G)=\{0\}$, then there exists a least energy solution $\bf w=(w_1,\ldots,w_\ell)$ to the system \eqref{Sl_pi} such that, after passing to a subsequence,
			\[\lim_{k \to \infty}\|u_{ik}-w_{i}\|=0\qquad\text{for all \ }i=1,\ldots,\ell,
			\]
			and \ $c_\infty^\ell>\displaystyle\displaystyle\sum_{i=1}^\ell c_i$. Moreover, if $u_{ik}\geq 0$ for all $k\in\n$, then $w_i\geq 0$.
			\item[$(II)$] If $\F(G)$ has positive dimension, then, for each $i=1,\ldots,\ell$, there exist a sequence $(\xi_{ik})$ in $\F(G)$ and a least energy solution $v_i$ to the problem \eqref{P_pi} such that, after passing to a subsequence,
			\[\lim_{k \to \infty}|\xi_{ik}-\xi_{jk}|=\infty \text{ if }i\neq j,\qquad\lim_{k\to\infty}\|u_{ik}-v_{i}( \ \cdot \ -\xi_{ik})\|=0\quad\text{for \ }i,j=1,\ldots,\ell,\]
			and \ $c_\infty^\ell=\displaystyle\sum_{i=1}^\ell c_i$.  Moreover, if $u_{ik}\geq 0$ for all $k\in\n$, then $v_i\geq 0$.
		\end{itemize}
	\end{theorem}
	
	\begin{proof}
		We split the proof into three steps.
		\medskip
		
		\textsc{Step 1.} \ \emph{The choice of the concentration points.}
		
		As $\bf u_k\in\cN_\infty^\ell$ and $\lambda_{ij}<0$, from Lemma \ref{N} we get a constant $c_0>0$ such that
		\begin{equation} \label{eq:lions}
			\displaystyle\irn\mu_i|u_{ik}|^p>c_0\qquad\text{for all \ }k\in\n \text{ \ and \ } i=1,\ldots,\ell.
		\end{equation}
		Invoking \cite[Theorem 3.3]{CSz} and Ekeland's variational principle \cite[Theorem 8.5]{W} we may assume without loss of generality that $(\cJ_\infty^\ell)'(\bf u_k)\to 0$ in the dual space $(\cH^\ell)'$. From \eqref{eq:lions} and Lions' lemma \cite[Lemma 1.21]{W} we see that there exists $\theta>0$ such that
		\begin{equation*}
			\limsup_{k \to \infty}\sup_{y \in \rn}\int_{B_1(y)}\mu_i|u_{ik}|^p>2\theta>0\qquad\text{for all \ }i=1,\ldots,\ell.
		\end{equation*}
		So, passing to a subsequence, we may choose $y_{ik}\in\rn$ satisfying
		\begin{equation}\label{eeMT1}
			\int_{B_1(y_{ik})}\mu_i|u_{ik}|^p > \theta \qquad\text{for all \ }k\in\n \text{ \ and \ } i=1,\ldots,\ell.
		\end{equation}
		Applying Lemma \ref{lem:orbits} to $(y_{ik})$ we obtain a sequence $(\zeta_{ik})$ and a constant $C$ such that $\mathrm{dist}(G y_{ik},\zeta_{ik})\leq C$ for all $k$ and $i$. Since $|u_{ik}|$ is $G$-invariant, taking $g_{ik}\in G$ such that $|g_{ik}y_{ik}-\zeta_{ik}|=\mathrm{dist}(Gy_{ik},\zeta_{ik})$ we get
		\begin{equation} \label{eq:nonull}
			\int_{B_{C+1}(\zeta_{ik})}\mu_i|u_{ik}|^p\geq\int_{B_1(g_{ik}y_{ik})}\mu_i|u_{ik}|^p\geq\int_{B_1(y_{ik})}\mu_i|u_{ik}|^p > \theta>0.
		\end{equation}
		We claim that $\zeta_{ik}\in\F(G)$. Otherwise, by the alternative in Lemma \ref{lem:orbits}, for every $M\in\n$ we can choose $g_{i1},\ldots,g_{iM}\in G$ such that
		$$|g_{in} \zeta_{ik} - g_{in'} \zeta_{ik}| \geq 2(C+1) \qquad\text{for all pairs \ }n\neq n', \ \ n,n' = 1,\ldots,M,$$
		and for sufficiently large $k \in \n$. As a consequence, $B_{C+1}(g_{in}\zeta_{ik})\cap B_{C+1}(g_{in'}\zeta_{ik})=\emptyset$ and, so, 
		\begin{equation*}
			\displaystyle\irn\mu_i|u_{ik}|^p\geq\displaystyle\sum_{n=1}^M\int_{B_{C+1}(g_{in}\zeta_{ik})}\mu_i|u_{ik}|^p=M\int_{B_{C+1}(\zeta_{ik})}\mu_i|u_{ik}|^p> M\theta,
		\end{equation*}
		which is a contradiction, because $M$ is arbitrary and, as $(\bf u_k)$ is bounded in $\cH^\ell$, the sequence $(u_{ik})$ is bounded in $L^p(\rn)$. This shows that $\zeta_{ik}\in\F(G)$.
		
		After passing to a subsequence and reordering the components, there are integers \[{0=:\ell_0<\ell_1<\cdots<\ell_n:=\ell}\] such that, setting $I_r:=(\ell_r,\ell_{r+1}]$ for $r=0,\ldots,n-1$, we have
		\begin{align*}
			&(|\zeta_{ik}-\zeta_{jk}|)\text{ is bounded} &&\text{if \ }i,j\in I_r\text{ for some } r, \\
			&|\zeta_{ik}-\zeta_{jk}|\to\infty &&\text{if \ }i\in I_r\text{ and }j\in I_{r'}\text{ with }r\neq r'.
		\end{align*}
		Now we replace $\zeta_{ik}$ by $\xi_{ik}:=\zeta_{\ell_{r+1}\,k}$ for every $i\in I_r$, and we fix $R>0$ large enough so that $B_{C+1}(\zeta_{ik})\subset B_R(\xi_{ik})$ for all $k\in\n$ and $i=1,\ldots,\ell$. Then these new points satisfy
		\begin{align*}
			&\xi_{ik}\in\F(G), \\
			&\xi_{ik}=\xi_{jk} \text{ for all }k &&\text{if \ }i,j\in I_r\text{ for some } r, \\
			&|\xi_{ik}-\xi_{jk}|\to\infty &&\text{if \ }i\in I_r\text{ and }j\in I_{r'}\text{ with }r\neq r',
		\end{align*}
		and from \eqref{eq:nonull} we obtain
		\begin{equation} \label{eq:nonull2}
			\int_{B_R(\xi_{ik})}\mu_i|u_{ik}|^p > \theta>0\qquad\text{for all \ }k\in\n \text{ \ and \ } i=1,\ldots,\ell.
		\end{equation}
		Note that, if $\F(G)=\{0\}$, then $\xi_{ik}=0$ for all $k$ and $i$, and there is only one interval $I_0:=(0,\ell\,]$.
		\medskip
		
		\textsc{Step 2.} \ \emph{A solution to a subsystem.}
		
		Define
		$$w_{ik}(x):=u_{ik}(x+\xi_{ik}).$$
		Observe that $w_{ik}\in H^1(\rn)^{\phi_i}$, because $\xi_{ik}\in\F(G)$. Since the sequence $(w_{ik})$ is bounded in $H^1(\rn)$, passing to a subsequence, we have that $w_{ik}\rh w_i$ weakly in $H^1(\rn)$, $w_{ik}\to w_i$ in $L^p_\mathrm{loc}(\rn)$ and  $w_{ik}\to w_i$ a.e. in $\rn$. Then, $w_i\geq 0$ if $u_{ik}\geq 0$ for all $k\in\n$. As, by \eqref{eq:nonull2},
		\begin{equation*}
			\int_{B_R(0)}\mu_i|w_{ik}|^p=\int_{B_R(\xi_{ik})}\mu_i|u_{ik}|^p > \theta>0\qquad\text{for all \ }k\in\n \text{ \ and \ } i=1,\ldots,\ell,
		\end{equation*}
		we deduce that $w_i\neq 0$ for all $i=1,\ldots,\ell$. Furthermore,
		\begin{align} \label{eq:upperbd}
			\frac{p-2}{2p}\sum_{i=1}^\ell\|w_i\|^2&\leq\lim_{k\to\infty}\frac{p-2}{2p}\sum_{i=1}^\ell\|w_{ik}\|^2=\lim_{k\to\infty}\frac{p-2}{2p}\sum_{i=1}^\ell\|u_{ik}\|^2=\lim_{k\to\infty}\cJ_\infty^\ell(\bf u_k)=c_\infty^\ell.
		\end{align}
		
Fix $r\in\{0,\ldots,n-1\}$, take $\xi_{ik}$ with $i\in I_r$ and define $\bf u_k^r=(u^r_{1k},\ldots,u^r_{\ell k})$ by 
		\begin{equation*}
			u^r_{jk}(x):=u_{jk}(x+\xi_{ik})=
			\begin{cases}
				w_{jk}(x) &\text{if \ }j\in I_r,\\
				w_{jk}(x+\xi_{ik}-\xi_{jk}) &\text{if \ }j\not\in I_r.
			\end{cases}
		\end{equation*}
After passing to a subsequence, we have that $u^r_{jk}\rh u^r_j$ weakly in $H^1(\rn)$. Note that $u_j^r=w_j\neq 0$ if $j\in I_r$. Set $\bf u^r:=(u_1^r,\ldots,u_\ell^r)$. Let $\vp\in\cC^\infty_c(\rn)$ and define $\widehat{\vp}_{k}(x):=\vp(x-\xi_{ik})$. Then we have that $\partial_j\cJ^\ell_\infty(\bf u_k)\widehat{\vp}_{k}=\partial_j\cJ^\ell_\infty(\bf u_k^r)\vp$ and, as $(\cJ_\infty^\ell)'(\bf u_k)\to 0$, we deduce that
\begin{align*}
0=\lim_{k\to\infty}\partial_j\cJ^\ell_\infty(\bf u_k)\widehat{\vp}_{k}=\lim_{k\to\infty}\partial_j\cJ^\ell_\infty(\bf u_k^r)\varphi=\partial_j\cJ^\ell_\infty(\bf u^r)\varphi\qquad\text{for all \ }j=1,\ldots,\ell.
\end{align*}
This shows that $\bf u^r$ is a critical point of $\cJ_\infty^\ell:\cH^\ell\to\r$. Hence, its nontrivial components give a fully nontrivial solution to the subsystem
		\begin{equation} \label{eq:subsystem}
			\begin{cases}
				-\Delta v_m+v_m=\mu_m|v_m|^{p-2}v_m + \sum\limits_{\substack{j\in \widehat{I}_r \\ j \not=m}}\lambda_{mj}\beta_{mj}|v_j|^{\alpha_{mj}}|v_m|^{\beta_{mj} -2}v_m,\\
				v_m \in H^1(\rn)^{\phi_m}, \qquad v_m\neq 0, \qquad m\in \widehat{I}_r,
			\end{cases}
		\end{equation}
where $\widehat{I}_r:=\{m:1\leq m\leq\ell, \ u^r_m\neq 0\}$. Note that $\widehat{I}_r\supset I_r$.
\medskip
		
		\textsc{Step 3.} \ \emph{The conclusion.}
		
		Now we distinguish two cases.
		
		$(I):$ Let $\F(G)=\{0\}$. Then $\xi_{ik}=0$ and $w_{ik}=u_{ik}$ for all $k$, $\widehat{I}_0=\{1,\ldots,\ell\}$ and ${\bf w:=(w_1,\ldots,w_\ell)}$ is a fully nontrivial solution to the system \eqref{eq:subsystem} with $r=0$. From \eqref{eq:upperbd} we derive
\begin{align*} 
c_\infty^\ell\leq\cJ_\infty^\ell(\bf w)=\frac{p-2}{2p}\displaystyle\sum_{i=1}^\ell\|w_i\|^2\leq \lim_{k\to\infty}\frac{p-2}{2p}\displaystyle\sum_{i=1}^\ell\|u_{ik}\|^2=c_\infty^\ell.
\end{align*}
As $u_{ik}\rh w_i$ weakly in $H^1(\rn)$, this implies that $u_{ik}\to w_i$ strongly in $H^1(\rn)$ and that 
		$\cJ_\infty^\ell(\bf w)=c_\infty^\ell$. Then, by Proposition \ref{prop:energy_estimates}$(iii)$, we have that $c_\infty^\ell>\displaystyle\sum_{i=1}^\ell c_i$, as claimed.
		
		$(II):$ Let $\F(G)$ have positive dimension. Since for each $r\in\{0,\ldots,n-1\}$, the nontrivial components of $\bf u^r$ satisfy \eqref{eq:subsystem}, multiplying the $i$-th equation by $w_i$ we get
\begin{equation} \label{eq:w_i}
\|w_i\|^2=\irn\mu_i|w_i|^p+\sum\limits_{\substack{j\in \widehat{I}_r \\ j \not=i}}\irn\lambda_{ij}\beta_{ij}|u^r_j|^{\alpha_{ij}}|w_i|^{\beta_{ij}}\leq \irn\mu_i|w_i|^p.
\end{equation}
Hence, there exists $t_i\in (0,1]$ such that $t_iw_i\in\cN_i$ and, so,
$$c_i\leq\frac{p-2}{2p}\|t_iw_i\|^2\leq\frac{p-2}{2p}\|w_i\|^2\qquad\text{for every \ }i=1,\ldots,\ell.$$
From \eqref{eq:upperbd} and Proposition \ref{prop:energy_estimates} we derive
\begin{equation} \label{eq:upperbd2}
\sum_{i=1}^\ell c_i\leq\frac{p-2}{2p}\sum_{i=1}^\ell\|w_i\|^2\leq\lim_{k\to\infty}\frac{p-2}{2p}\sum_{i=1}^\ell\|w_{ik}\|^2=c_\infty^\ell=\sum_{i=1}^\ell c_i.
\end{equation}
It follows that $\frac{p-2}{2p}\|w_i\|^2=c_i$ and  $t_i=1$, so $w_i$ is a nontrivial least energy solution to the problem \eqref{P_pi} for each $i=1,\ldots,\ell$. Furthermore, \eqref{eq:w_i} yields
$$\irn\lambda_{ij}\beta_{ij}|u^r_j|^{\alpha_{ij}}|w_i|^{\beta_{ij}}=0\qquad\text{for every \ }j\neq i.$$
Therefore, $u^r_j=0$ for every $j\neq i$ and, so, $\widehat{I}_r=I_r=\{i\}$. This implies that $|\xi_{ik}-\xi_{jk}|\to\infty$ for all $j\neq i$. Finally, as $w_{ik}\rh w_i$ weakly in $H^1(\rn)$, it follows from \eqref{eq:upperbd2} that $w_{ik}\to w_i$ strongly in $H^1(\rn)$. This implies that $\displaystyle\lim_{k\to\infty}\|u_{ik}-w_{i}( \ \cdot \ -\xi_{ik})\|=0$, and the proof is complete.
	\end{proof}
	
	From Theorem \ref{thm:spliting} and Proposition \ref{prop:energy_estimates}$(ii)$-$(iii)$ we immediately obtain the following result.
	
	\begin{theorem} \label{thm:existence_rn}
		Let $\ell\geq 2$. The system \eqref{Sl_pi} has a least energy solution if and only if $\F(G)=\{0\}$.
	\end{theorem}
	
	Theorem \ref{thm:main_rn} follows from the previous one.
	
	\begin{proof}[Proof of Theorem \ref{thm:main_rn}] 
		The group $G:=\Gamma\times O(N-4)$ introduced in Example \ref{examples}$(iii)$ satisfies $(A_2)$ if either $N=4$ or $N\geq 6$, and $\F(G)=\{0\}$. For $i=1,\ldots,m$ we set $\phi_i\equiv 1$, and for $i=m+1,\ldots,\ell$ we take $\phi_i$ as in Example \ref{examples}$(iii)$. Applying Theorem \ref{thm:existence_rn} we obtain a least energy solution $\bf w=(w_1,\ldots,w_\ell)$ to the system \eqref{Sl_i} satisfying \eqref{eq:symmetries}. By Remark \ref{rem:positive}, we may take $w_1,\ldots,w_m$ to be positive. The symmetries ensure that the last $\ell-m$ components are nonradial and change sign.
	\end{proof}
	
	To further emphasize the role of the fixed points, we prove the following result.
	
	\begin{theorem} \label{thm:multiplicity}
		If $\F(G)=\{0\}$, then the system \eqref{Sl_pi} has infinitely many solutions.
	\end{theorem}
	
	The key ingredient is the following fact.
	
	\begin{lemma} \label{lem:compactness}
		If $G$ satisfies $(A_2)$ and $\F(G)=\{0\}$, then the embedding $H^1(\rn)^\phi\hookrightarrow L^p(\rn)$ is compact for any homomorphism $\phi:G\to\z_2$ and $p\in(2,2^*)$.
	\end{lemma}
	\begin{proof}
		Let $K:=\ker\phi$. Since the $K$-orbit of every $x\in\rn\smallsetminus\{0\}$ is infinite, the embedding $H^1(\rn)^K\hookrightarrow L^p(\rn)$ is compact, where $H^1(\rn)^K$ denotes the space of $K$-invariant functions in $H^1(\rn)$; see, e.g., \cite[Lemma 4.3]{CSz}. Since $H^1(\rn)^\phi\subset H^1(\rn)^K$, our claim follows.
	\end{proof}
	
	\begin{proof}[Proof of Theorem \ref{thm:multiplicity}]
		We note that the results in \cite[Section 3]{CSz} hold true if we take $\cH:=\cH^\ell$. A standard argument using Lemma \ref{lem:compactness} shows that $\cJ_\infty^\ell:\cH^\ell\to\r$ satisfies the Palais-Smale condition at every level $c\in\r$. As in \cite[Section 3]{CSz} we set
		$$\cU:=\{\bf u\in\cH^\ell:\|u_i\|=1\text{ \ for \ }i=1,\ldots,\ell,\text{ \ and there exists \ }\bf t\in(0,\infty)^\ell\text{ \ such that \ }\bf t\bf u\in\cN_\infty^\ell\},$$
		and we define $\Psi:\cU\to\r$ by $\Psi(u):=\cJ_\infty^\ell(\bf t\bf u)$ for $\bf t\in(0,\infty)^\ell$ such that $\bf t\bf u\in\cN_\infty^\ell$. Such $\bf t$ is unique. 
		
		It follows from \cite[Theorem 3.3]{CSz} that $\Psi\in\cC^1(\cU,\r)$ and it satisfies the Palais-Smale condition at every level $c\in\r$. Exactly the same argument used to prove \cite[Lemma 4.5]{CSz} shows that $\mathrm{genus}(\cU)=\infty$ (one simply replaces $H_0^1(\o)^G$ with $H^1(\rn)^{\phi_i}$). Now \cite[Theorem 3.4]{CSz} asserts that \eqref{Sl_pi} has an unbounded sequence of solutions.
	\end{proof}
	
	We observe that, while the last $\ell-m$ components of the solutions given by the Theorem \ref{thm:multiplicity} change sign by construction, this theorem does not give any information about the sign of the first $m$ components.

	\section{The asymptotic behavior of minimizers}	
	\label{sec:profiles}	
	
	As before, we fix a closed subgroup $G$ of $O(N)$ satisfying $(A_2)$ and, for each $i=1,\ldots,\ell$, a continuous homomorphism $\phi_i:G\to\z_2$ satisfying $(A_1)$. Furthermore, we assume throughout that $\o$ is a $G$-invariant bounded domain in $\rn$ and that $0\in\o$. 
	
	Our aim is to prove the following result.
	
	\begin{theorem} \label{thm:profiles}
		For any given sequence $(\eps_k)$ of positive numbers converging to zero, there exists a least energy solution $\bf u_k=(u_{1k},\ldots,u_{\ell k})$ to the system $(\mathscr S^\phi_{\eps_k,\o})$ with the following limit profile:
		\begin{itemize}
			\item[$(I)$] If $\F(G)=\{0\}$, then there exists a least energy solution $\bf w=(w_1,\ldots,w_\ell)$ to the system \eqref{Sl_pi} such that, after passing to a subsequence,
			\[\lim_{k \to \infty}\|u_{ik}-w_{i}(\eps_k^{-1} \ \cdot \ )\|_{\eps_k}=0\qquad\text{for all \ }i=1,\ldots,\ell,
			\]
			and \ $\displaystyle\lim_{\eps \to 0}c_\eps^\ell(\o) =c_\infty^\ell$. Moreover, $w_i\geq 0$ if $u_{ik}\geq 0$ for all $k\in\n$.
			\item[$(II)$] If $\F(G)$ has positive dimension, then, for each $i=1,\ldots,\ell$, there exist a sequence $(\xi_{ik})$ in $\o\cap\F(G)$ and a least energy solution $v_i$ to the problem \eqref{P_pi} such that, after passing to a subsequence,
			\[\lim_{k \to \infty}\eps_k^{-1} \mathrm{dist}(\zeta_{ik},\partial\o) = \infty,\quad\lim_{k \to \infty}\eps_k^{-1}|\zeta_{ik}-\zeta_{jk}|=\infty\text{ if }i\neq j,\qquad\lim_{k\to\infty}\|u_{ik}-v_{i}(\eps_k^{-1}( \ \cdot \ -\zeta_{ik}))\|_{\eps_k}=0,\]
			for $i,j=1,\ldots,\ell$, \ and \ $\displaystyle\lim_{\eps \to 0}c_\eps^\ell(\o) =\displaystyle\sum_{i=1}^\ell c_i$. Moreover, $v_i\geq 0$ if $u_{ik}\geq 0$ for all $k\in\n$.
		\end{itemize}
	\end{theorem}
	
	We start with the proof of the following statement.
	
	\begin{lemma}\label{lem:limit_energy} 
		$\displaystyle\lim_{\eps \to 0}c_\eps^\ell(\o) = c_\infty^\ell$.
	\end{lemma}
	
	\begin{proof} 
		Let $\bf u=(u_1,\ldots,u_\ell)\in \cN^\ell_\eps(\o)$ and define $\tilde u_i(z):=u_i(\eps z)$. By Remark \ref{rem:norms}, $\tilde{\bf u}=(\tilde u_1,\ldots,\tilde u_\ell)\in \cN^\ell_\infty$ and
		$\cJ_\eps^\ell(\bf u)=\cJ_\infty^\ell(\tilde{\bf u})\geq c_\infty^\ell$. Therefore, 
		\begin{equation}\label{e0P3.10}
			c_\infty^\ell\leq \displaystyle\liminf_{\eps\to 0}c_\eps^\ell.
		\end{equation} 
		To prove the opposite inequality, let $\bf w=(w_1,\ldots,w_\ell) \in \cN^\ell_\infty$. Fix $r>0$ such that $B_r(0)\subset\o$ and a radial cut-off function $\chi \in C^\infty_c(\mathbb{R}^N)$ satisfying $\chi(x) = 1$ for $|x| \leq \frac{r}{2}$ and $\chi(x) = 0$ for $|x|>r$. For $\eps>0$ define $\bf w_\eps=(w_{1\eps},\ldots,w_{\ell\eps})$ by
		$$w_{i\eps}(x) :=  w_i\Big(\frac{x}{\eps}\Big)\chi(x)\qquad i =1,\ldots,\ell.$$
		Since $w_i\in H^1(\rn)^{\phi_i}$ and $\chi$ is radial, we have that $w_{i\eps}\in H^1(\rn)^{\phi_i}$. As $\bf w\in \cN^\ell_\infty$, from \eqref{eq:nehari} and Lemma \ref{N}$(i)$ we get
		$$0<c_0\leq \|w_{i}\|^2 = {\mu_i}|w_{i}|_p^{p} + \displaystyle\sum_{\substack{j=1\\j\not=i}}^\ell\int_{\mathbb{R}^N}{\lambda_{ij}}\beta_{ij}|w_{j}|^{\alpha_{ij}}|w_{i}|^{\beta_{ij}} \qquad \text{for all \ }i =1,\ldots,\ell.$$
		Note that $w_{i\eps} \to w_i$ in $H^1(\rn)$ as $\eps \to 0$. Therefore, for sufficiently small $\eps>0$, Lemma \ref{N}$(ii)$ yields $\bf t_\eps = (t_{1\eps},\ldots,t_{\ell\eps})\in (0,\infty)^\ell$ such that $\bf t_\eps \bf w_\eps \in \cN^\ell_\eps(\o)$, and $t_{i\eps}\to 1$ as $\eps\to 0$. Thus,
		$$c_\eps^\ell(\o) \leq  \cJ^\ell_\eps(\bf t_\eps \bf w_\eps) = \dfrac{p-2}{2p}\|\bf t_\eps \bf w_\eps\|_{\ell,\eps}^2 \to \dfrac{p-2}{2p}\| \bf w\|_{\ell}^2 =\cJ^\ell_\infty(\bf w) \quad \text{as} \quad \eps \to 0,$$
		and, as a consequence,
		\begin{equation}\label{e3P3.10}
			\limsup_{\eps \to 0}c_\eps^\ell(\o) \leq c_\infty^\ell.
		\end{equation}
		Inequalities \eqref{e0P3.10} and \eqref{e3P3.10} yield the desired identity.
	\end{proof}
	
	\begin{proof}[Proof of Theorem \ref{thm:profiles}]
		Applying Theorem \ref{thm:minimizers} we obtain a least energy solution $\bf u_k=(u_{1k},\ldots,u_{\ell k})$ to the system $(\mathscr S^\phi_{\eps_k,\o})$ for every $k\in\n$. Define $\tilde{\bf u}_k=(\tilde u_{1k},\ldots,\tilde u_{\ell k})$ by taking $\tilde u_{ik}(z):=u_{ik}(\eps_k z)$. From Remark \ref{rem:norms} and Lemma \ref{lem:limit_energy} we see that $\tilde{\bf u}_k\in\cN_\infty^\ell$ and $\cJ_\infty^\ell(\tilde{\bf u}_k)\to c_\infty^\ell$. Now we distinguish two cases:
		
		$(I)$ If $\F(G)=\{0\}$, applying Theorem \ref{thm:spliting} to  $\tilde{\bf u}_k$ we obtain a least energy solution $\bf w=(w_1,\ldots,w_\ell)$ to the system \eqref{Sl_pi} such that $\|u_{ik}-w_{i}(\eps_k^{-1} \ \cdot \ )\|_{\eps_k}=\|\tilde u_{ik}-w_{i}\|\to 0$. The identity $\displaystyle\lim_{\eps \to 0}c_\eps^\ell(\o) = c_\infty^\ell$ was proved in Lemma \ref{lem:limit_energy}.
		
		$(II)$ If $\F(G)$ has positive dimension, Theorem \ref{thm:spliting} applied to  $\tilde{\bf u}_k$ yields a sequence $(\xi_{ik})$ in $\F(G)$ and a least energy solution $v_i$ to the problem \eqref{P_pi} for each $i=1,\ldots,\ell$, such that, setting $\zeta_{ik}:=\eps_k\xi_{ik}$, we get
		\[\lim_{k \to \infty}\eps_k^{-1}|\zeta_{ik}-\zeta_{jk}|=\infty\text{ \ if \ }i\neq j,\qquad\text{and}\qquad \|u_{ik}-v_{i}(\eps_k^{-1}( \ \cdot \ -\zeta_{ik}))\|_{\eps_k}=\|\tilde u_{ik}-v_{i}( \ \cdot \ -\xi_{ik})\|\to 0,\]
		for $i,j=1,\ldots,\ell$. We are left with showing that $\zeta_{ik}\in\o$ and
		\[\lim_{k \to \infty}\eps_k^{-1} \mathrm{dist}(\zeta_{ik},\partial\o) = \infty\qquad\text{for all \ }i=1,\ldots,\ell.\]
		To this end, we recall that the points $\xi_{ik}$ satisfy \eqref{eq:nonull2}, i.e., there exist $\theta,R>0$ such that
		\begin{equation} \label{eq:nonull3}
			\int_{B_R(\xi_{ik})}\mu_i|\tilde u_{ik}|^p > \theta>0\qquad\text{for all \ }k\in\n \text{ \ and \ } i=1,\ldots,\ell.
		\end{equation}
		Note that $\supp(\tilde u_{ik})\subset\o_{\eps_k}$, where $\o_{\eps_k}:=\{z\in\rn:\eps_kz\in\o\}$. So \eqref{eq:nonull3} implies that $\mathrm{dist}(\xi_{ik},\o_{\eps_k})<R$ for every $k\in\n$ and every $i$. A standard argument shows that $\mathrm{dist}(\xi_{ik},\partial\o_{\eps_k})\to\infty$, otherwise $v_i$ would be zero in some open subset of $\rn$, contradicting the unique continuation principle. In particular, we have that $\xi_{ik}\in\o_{\eps_k}$ for all $k\in\n$. Rescaling yields $\zeta_{ik}\in\o$ and 
		$\displaystyle\lim_{k \to \infty}\eps_k^{-1} \mathrm{dist}(\zeta_{ik},\partial\o) = \infty$ for all $i=1,\ldots,\ell$, as claimed. This completes the proof.
	\end{proof}
	
	Next, we derive Theorems \ref{thm:main1} and \ref{thm:main2} from Theorem \ref{thm:profiles}.
	
	\begin{proof}[Proof of Theorem \ref{thm:main1}.] The group $G:=\Gamma\times O(N-4)$ in Example \ref{examples}$(iii)$ satisfies $(A_2)$ if $N=4$ or $N\geq 6$, and $\F(G)=\{0\}$. For $i=1,\ldots,m$ we take $\phi_i\equiv 1$, and for $i=m+1,\ldots,\ell$ we take $\phi_i:=\phi$ as in Example \ref{examples}$(iii)$. By Theorem \ref{thm:profiles} there exist a least energy solution $\widehat{\bf u}_k=(\widehat u_{1k},\ldots,\widehat u_{\ell k})$ to the system $(\mathscr S^\phi_{\eps_k,B_1(0)})$ and a least energy solution $\bf w=(w_1,\ldots,w_\ell)$ to the system \eqref{Sl_i} satisfying \eqref{eq:symmetries} such that, after passing to a subsequence,
		\[\lim_{k \to \infty}\|\widehat u_{ik}-w_{i}(\eps_k^{-1} \ \cdot \ )\|_{\eps_k}=0\qquad\text{for all \ }i=1,\ldots,\ell,\]
		and \ $\displaystyle\lim_{k\to\infty}\displaystyle\sum_{i=1}^\ell\|\widehat u_{ik}\|_{\eps_k}^2=\lim_{\eps \to 0}c_\eps^\ell(B_1(0))=c_\infty^\ell=\displaystyle\sum_{i=1}^\ell\|w_i\|^2$. By Remark \ref{rem:positive}, we may take $\widehat u_{1k},\ldots,\widehat u_{m k}$ to be positive. Then, $w_1,\ldots,w_m$ are also positive. The symmetries ensure that the last $\ell-m$ components of $\widehat{\bf u}_k$ and $\bf w$ are nonradial and change sign.
	\end{proof}
	\smallskip
	
	\begin{proof}[Proof of Theorem \ref{thm:main2}.] The group $G:=\Gamma$ in Example \ref{examples}$(ii)$ satisfies $(A_2)$, and $\F(G)=\{0\}\times\r^{N-4}$ has positive dimension if $N\geq 5$. For $i=1,\ldots,m$ we take $\phi_i\equiv 1$, and for $i=m+1,\ldots,\ell$ we take $\phi_i:=\phi$ as in Example \ref{examples}$(ii)$. By Theorem \ref{thm:profiles}, there exist a least energy solution $\bf u_k=(u_{1k},\ldots,u_{\ell k})$ to the system $(\mathscr S^\phi_{\eps_k,B_1(0)})$ and, for each $i=1,\ldots,\ell$, a sequence $(\zeta_{ik})$ in $B_1(0)\cap\F(G)$ and a least energy solution $w_i$ to the problem \eqref{P_pi} such that, after passing to a subsequence,
		\[\lim_{k \to \infty}\eps_k^{-1} \mathrm{dist}(\zeta_{ik},\partial B_1(0)) = \infty,\quad\lim_{k \to \infty}\eps_k^{-1}|\zeta_{ik}-\zeta_{jk}|=\infty\text{ if }i\neq j,\qquad\lim_{k\to\infty}\|u_{ik}-w_{i}(\eps_k^{-1}( \ \cdot \ -\zeta_{ik}))\|_{\eps_k}=0,\]
		for $i,j=1,\ldots,\ell$, \ and \ $\displaystyle\lim_{k\to\infty}\displaystyle\sum_{i=1}^\ell\|u_{ik}\|_{\eps_k}^2=\lim_{\eps \to 0}c_\eps^\ell(B_1(0)) =\displaystyle\sum_{i=1}^\ell c_i=\displaystyle\sum_{i=1}^\ell\|w_i\|^2$. 
		
		By Remark \ref{rem:positive}, we may take $u_{1k},\ldots,u_{m k}$ to be positive. Then, $w_1,\ldots,w_m$ are also positive. It is well known that problem \eqref{P_i} has a unique positive radial solution $v_i$ and that any other positive solution is a translation of it \cite{k}. Hence, there exists $\vartheta_i\in\rn$ such that $v_i(y)=w_i(y+\vartheta_i)$ for every $y\in\rn$ and $i=1,\ldots,m$.
		
		Now take $i\in\{m+1,\ldots,\ell\}$. Then, $w_i\in H^1(\rn)^\phi$ with $\phi$ as in Example \ref{examples}$(ii)$, i.e., $w_i$ satisfies
		\begin{equation} \label{eq:symmetries3}
			w_i(z_1,z_2,x)=w_i(\mathrm{e}^{\mathrm{i}\vartheta} z_1,\mathrm{e}^{\mathrm{i}\vartheta} z_2,x)\text{ \ for all \ }\vartheta\in[0,2\pi),\qquad w_i(z_1,z_2,x)=-w_i(z_2,z_1,x),
		\end{equation}
		for all $(z_1,z_2,x)\in\cc\times\cc\times\r^{N-4}$. We claim that there exists $\vartheta_i\in\F(G)$ such that $v_i(y):=w_i(y+\vartheta_i)$ satisfies \eqref{eq:symmetries2}. To prove this claim, observe that, if $\pi_j:\rn\to\rn$ is the reflection on the hyperplane $\{(z,x_1,\ldots,x_{N-4})\in\cc^2\times\r^{N-4}:x_j=a\}$, $a\in\r$, then the function
		\begin{equation*}
			\bar w(z,x):=
			\begin{cases}
				w_i(z,x) & \text{if \ }x_j>a,\\
				(w_i\circ\pi_j)(z,x) & \text{if \ }x_j<a,
			\end{cases}
		\end{equation*}
		satisfies \eqref{eq:symmetries3}. With this remark and following Lopes' argument \cite{l} (see also \cite[Appendix C]{W}) one shows that, for each $j=1,\ldots,N-4$, there exists $\theta_{ij}\in\r$ such that $w_i$ is invariant with respect to the reflection on the hyperplane $x_j=\theta_{ij}$. Setting $\theta_i=(\theta_{i1},\ldots,\theta_{1\,N-4})$ we infer that the function $v_i(z,x):=w_i(z,x+\theta_i)$ satisfies $v_i(z,x)=v_i(z,gx)$ for all $g\in O(N-4)$ and $(z,x)\in\cc^2\times\r^{N-4}$. Hence, setting $\vartheta_i=(0,\theta_i)\in\cc^2\times\r^{N-4}$, we have that $v_i(y)=w_i(y+\vartheta_i)$ satisfies \eqref{eq:symmetries2}, as claimed.
		
		Let $\xi_{ik}:=\zeta_{ik}+\eps_k\vartheta_i$, where $\vartheta_i$ are the points defined in the previous paragraphs, so that $v_i(y):=w_i(y+\vartheta_i)$ is radial if $i=1,\ldots,m$ and it satisfies \eqref{eq:symmetries2} if $i=m+1,\ldots,\ell$. Note that, as $\eps_k\to 0$, we have that $\xi_{ik}\in B_1(0)$ for $k$ large enough. From the corresponding statements for $\zeta_{ik}$ and $w_i$ one immediately derives
		\[\lim_{k \to \infty}\eps_k^{-1} \mathrm{dist}(\xi_{ik},\partial B_1(0)) = \infty,\quad\lim_{k \to \infty}\eps_k^{-1}|\xi_{ik}-\xi_{jk}|=\infty\text{ if }i\neq j,\qquad\lim_{k\to\infty}\|u_{ik}-v_{i}(\eps_k^{-1}( \ \cdot \ -\xi_{ik}))\|_{\eps_k}=0,\]
		for $i,j=1,\ldots,\ell$, \ and \ $\displaystyle\lim_{k\to\infty}\displaystyle\sum_{i=1}^\ell\|u_{ik}\|_{\eps_k}^2=\lim_{\eps \to 0}c_\eps^\ell(B_1(0)) =c_\infty^\ell=\displaystyle\sum_{i=1}^\ell c_i=\displaystyle\sum_{i=1}^\ell\|v_i\|^2$. 
		
		The inequality $\mathfrak{c}_m<\widehat{\mathfrak{c}}_m$ follows from Proposition \ref{prop:energy_estimates}.
	\end{proof}
	
	Finally, regarding positive solutions in the the unit ball $B_1(0)$ in $\rn$, we derive the following result.
	
	\begin{theorem} \label{thm:radial}
		Let $N\geq 2$, $\o=B_1(0)$ and  $(\eps_k)$ be a sequence of positive numbers converging to zero. Then there exist two solutions $\widehat{\bf u}_k=(\widehat u_{1k},\ldots,\widehat u_{\ell k})$ and $\bf u_k=(u_{1k},\ldots,u_{\ell k})$ to the system $(\mathscr S_{\eps_k,\o})$ whose components $\widehat u_{ik}$ and $u_{ik}$ are positive and satisfy:
		\begin{itemize}
			\item  $\widehat u_{ik}$ is radial and, after passing to a subsequence,
			\[\lim_{k \to \infty}\|\widehat u_{ik}-\widehat{\omega}_{i}(\eps_k^{-1} \ \cdot \ )\|_{\eps_k}=0\qquad\text{for all \ }i=1,\ldots,\ell,\]
			where $\widehat{\boldsymbol{\omega}}_k=(\widehat \omega_{1k},\ldots,\widehat\omega_{\ell k})$ is a least energy positive radial solution to the system \eqref{Sl_i} in $\rn$.
			\item  For each $i=1,\ldots,\ell$ there is a sequence $(\xi_{ik})$ in $B_1(0)$ such that, after passing to a subsequence,
			\[\lim_{k \to \infty}\eps_k^{-1} \mathrm{dist}(\zeta_{ik},\partial\o) = \infty,\quad\lim_{k \to \infty}\eps_k^{-1}|\zeta_{ik}-\zeta_{jk}|=\infty\text{ if }i\neq j,\qquad\lim_{k\to\infty}\|u_{ik}-v_{i}(\eps_k^{-1}( \ \cdot \ -\zeta_{ik}))\|_{\eps_k}=0,\]
			for $i,j=1,\ldots,\ell$, where $v_i$ is the positive radial solution to the problem \eqref{P_i}.
		\end{itemize}
	\end{theorem}
	
	\begin{proof}
		The solutions $\widehat{\bf u}_k$ are obtained applying Theorem \ref{thm:profiles} with $G=O(N)$ and $\phi_i\equiv 1$ for all $i=1,\ldots,\ell$. The solutions $\bf u_k=(u_{1k},\ldots,u_{\ell k})$ are obtained applying Theorem \ref{thm:profiles} to the trivial group $G=\{1\}$ and $\phi_i\equiv 1$ for all $i=1,\ldots,\ell$. 
	\end{proof} 

The solutions $\bf u_k$ were obtained and fully described in \cite{LW}. In addition, as mentioned in the introduction, the existence of a least energy radial solution to the system \eqref{Sl_i} was shown by Sirakov in \cite{s}. Theorem \ref{thm:radial} gives an alternative proof of this fact. Finally, we point out that the existence of multiple positive solutions, both in $\rn$ and in bounded domains, has been established, for instance, in \cite{bdw,dww,nr,ww1,ww2}. Positive solutions of singularly perturbed systems with potentials have been obtained, e.g., in \cite{MPS,p}.

	\medskip
	
	\begin{flushleft}
		\textbf{Mónica Clapp} and \textbf{Mayra Soares}\\
		Instituto de Matemáticas\\
		Universidad Nacional Autónoma de México\\
		Circuito Exterior, Ciudad Universitaria\\
		04510 Coyoacán, Ciudad de México, México\\
		\texttt{monica.clapp@im.unam.mx \\ ssc\_mayra@im.unam.mx} 
	\end{flushleft}	
	

\begin{thebibliography}{99}
	
	\bibitem{bdw}Bartsch, Thomas; Dancer, Norman; Wang, Zhi-Qiang: A Liouville theorem, a-priori bounds, and bifurcating branches of positive solutions for a nonlinear elliptic system. Calc. Var. Partial Differential Equations 37 (2010), no. 3-4, 345–361. 
		
		\bibitem{BW} Bartsch, Thomas; Willem, Michel: Infinitely many nonradial solutions of a Euclidean scalar field equation. J. Funct. Anal. 117 (1993), no. 2, 447–460.
		
	
		
		\bibitem{JMW} Bracho, Javier; Clapp, Mónica; Marzantowicz, Wacław: Symmetry breaking solutions of nonlinear elliptic systems. Topol. Methods Nonlinear Anal. 26 (2005), no. 1, 189–201. 
		
			\bibitem{BjLyMsh}  Byeon, Jaeyoung; Lee, Youngae; Moon, Sang-Hyuck: Partly clustering solutions of nonlinear Schrödinger systems with mixed interactions. J. Funct. Anal. 280 (2021), no. 12, Paper No. 108987, 54 pp.
		
		\bibitem{clz} Chen, Zhijie; Lin, Chang-Shou; Zou, Wenming: Infinitely many sign-changing and semi-nodal solutions for a nonlinear Schrödinger system. Ann. Sc. Norm. Super. Pisa Cl. Sci. (5) 15 (2016), 859–897.
		
		\bibitem{CC} Cingolani, Silvia; Clapp, Mónica; Secchi, Simone: Multiple solutions to a magnetic nonlinear Choquard equation. Z. Angew. Math. Phys. 63 (2012), no. 2, 233–248.
		
		\bibitem{CSr} Clapp, Mónica; Srikanth, P. N.: Entire nodal solutions of a semilinear elliptic equation and their effect on concentration phenomena. J. Math. Anal. Appl. 437 (2016), no. 1, 485–497. 
		
		\bibitem{CSz} Clapp, Mónica; Szulkin, Andrzej: A simple variational approach to weakly coupled competitive elliptic systems. NoDEA Nonlinear Differential Equations Appl. 26 (2019), no. 4, Paper No. 26, 21 pp.
		
		\bibitem{dww}Dancer, E. N.; Wei, Juncheng; Weth, Tobias: A priori bounds versus multiple existence of positive solutions for a nonlinear Schrödinger system. Ann. Inst. H. Poincaré Anal. Non Linéaire 27 (2010), no. 3, 953–969. 
		
		\bibitem{BBEG} Esry, B. D.; Greene, Chris H.; Burke, Jr., James P.; Bohn, John L: Hartree-Fock theory for double condensates, {Phys. Rev. Lett.} {78} (1997), 3594-3597.
		
		\bibitem{k} Kwong, Man Kam: Uniqueness of positive solutions of $\Delta u-u+u^p=0$ in $\r^n$. Arch. Rational Mech. Anal. 105 (1989), no. 3, 243–266.
		
		\bibitem{LW} Lin, Tai-Chia; Wei, Juncheng: Spikes in two coupled nonlinear Schrödinger equations. Ann. Inst. H. Poincaré Anal. Non Linéaire 22 (2005), no. 4, 403–439.
		
		\bibitem{LW2}  Lin, Tai-Chia; Wei, Juncheng: Ground state of $N$ coupled nonlinear Schrödinger equations in $\r^n$, $n\leq3$. Comm. Math. Phys. 255 (2005), no. 3, 629–653.
		
		\bibitem{LLW} Liu, Jiaquan; Liu, Xiangqing; Wang, Zhi-qiang: Multiple mixed states of nodal solutions for nonlinear Schrödinger systems.	Calc. Var. Partial Differential Equations 52 (2015), no. 3-4, 565–586. 
		
		\bibitem{l} Lopes, Orlando: Radial symmetry of minimizers for some translation and rotation invariant functionals. J. Differential Equations 124 (1996), no. 2, 378–388.
		
		\bibitem{LU}  Lorca, Sebastián; Ubilla, Pedro: Symmetric and nonsymmetric solutions for an elliptic equation on $\rn$. Nonlinear Anal. 58 (2004), no. 7-8, 961–968.
		
		\bibitem{MPS}  Montefusco, Eugenio; Pellacci, Benedetta; Squassina, Marco: Semiclassical states for weakly coupled nonlinear Schrödinger systems. J. Eur. Math. Soc. (JEMS) 10 (2008), no. 1, 47–71.
		
		\bibitem{MPW} Musso, Monica; Pacard, Frank; Wei, Juncheng: Finite-energy sign-changing solutions with dihedral symmetry for the stationary nonlinear Schrödinger equation. J. Eur. Math. Soc. (JEMS) 14 (2012), no. 6, 1923–1953.
		
		\bibitem{nr} Noris, Benedetta; Ramos, Miguel: Existence and bounds of positive solutions for a nonlinear Schrödinger system. Proc. Amer. Math. Soc. 138 (2010), no. 5, 1681–1692. 
		
		\bibitem{p} Pomponio, Alessio: Coupled nonlinear Schrödinger systems with potentials. J. Differential Equations 227 (2006), no. 1, 258–281. 
		
		\bibitem{sw1} Sato, Yohei; Wang, Zhi-Qiang: On the multiple existence of semi-positive solutions for a nonlinear Schrödinger system. Ann. Inst. H. Poincaré Anal. Non Linéaire 30 (2013), no. 1, 1–22
		
		\bibitem{sw2}  Sato, Yohei; Wang, Zhi-Qiang: On the least energy sign-changing solutions for a nonlinear elliptic system. Discrete Contin. Dyn. Syst. 35 (2015), no. 5, 2151–2164. 
		
		\bibitem{s} Sirakov, Boyan: Least energy solitary waves for a system of nonlinear Schrödinger equations in Rn. Comm. Math. Phys. 271 (2007), no. 1, 199–221.
		
		\bibitem{tt} Tavares, Hugo; Terracini, Susanna: Sign-changing solutions of competition-diffusion elliptic systems and optimal partition problems. Ann. Inst. H. Poincaré Anal. Non Linéaire 29 (2012), no. 2, 279–300. 
		
		\bibitem{WWL} Wang, Weichung; Wu, Tsung-Fang; Liu, Chien-Hsiang: On the multiple spike solutions for singularly perturbed elliptic systems, Discrete Continuous Dynamical Systems Series B 18 (2013), no 1, 237-258.
		
		\bibitem{WjWy} Wei, Juncheng; Wu, Yuanze: Ground states of nonlinear Schrödinger systems with mixed couplings. J. Math. Pures Appl. (9) 141 (2020), 50–88. 
		
		
		
\bibitem{ww1}  Wei, Juncheng; Weth, Tobias: Nonradial symmetric bound states for a system of coupled Schrödinger equations. Atti Accad. Naz. Lincei Rend. Lincei Mat. Appl. 18 (2007), no. 3, 279–293. 	

\bibitem{ww2} Wei, Juncheng; Weth, Tobias: Radial solutions and phase separation in a system of two coupled Schrödinger equations. Arch. Ration. Mech. Anal. 190 (2008), no. 1, 83–106. 	
		
		\bibitem{W} Willem, Michel: Minimax theorems. Progress in Nonlinear Differential Equations and their Applications 24. Birkhäuser Boston, Inc., Boston, MA (1996).
		
		
		
		
	\end{thebibliography}
\end{document}